\documentclass[11pt]{amsart}
\usepackage{amsmath}
\usepackage{amssymb}
\usepackage{amsfonts}
\usepackage{latexsym}
\usepackage{amscd}
\usepackage[mathscr]{euscript}
\usepackage{xy} \xyoption{all}

\newcommand{\uloopr}[1]{\ar@'{@+{[0,0]+(-4,5)}@+{[0,0]+(0,10)}@+{[0,0] +(4,5)}}^{#1}}
\newcommand{\uloopd}[1]{\ar@'{@+{[0,0]+(5,4)}@+{[0,0]+(10,0)}@+{[0,0]+ (5,-4)}}^{#1}}
\newcommand{\dloopr}[1]{\ar@'{@+{[0,0]+(-4,-5)}@+{[0,0]+(0,-10)}@+{[0, 0]+(4,-5)}}_{#1}}
\newcommand{\dloopd}[1]{\ar@'{@+{[0,0]+(-5,4)}@+{[0,0]+(-10,0)}@+{[0,0 ]+(-5,-4)}}_{#1}}

\newcommand{\luloop}[1]{\ar@'{@+{[0,0]+(-8,2)}@+{[0,0]+(-10,10)}@+{[0, 0]+(2,2)}}^{#1}}

\newtheorem{lem}{Lemma}[section]

\newtheorem{theor}[lem]{Theorem}

\newtheorem{prop}[lem]{Proposition}
\theoremstyle{definition}
\newtheorem{defi}[lem]{Definition}

\newtheorem{exem}[lem]{Example}

\newcommand{\nc}{\newcommand}
\nc{\fin}{\mathrm{fin}}
\nc{\points}{\mathrm{pt }}
\nc{\iden}{\mathrm{id}}
\nc{\op}{\mathrm{op}}
\nc{\qcoh}{\mathrm{qcoh}}
\nc{\supph}{\mathrm{supph}}
\nc{\set}{\mathrm{set}}
\nc{\id}{\mathrm{id}}
\nc{\pr}{\mathrm{Pr}}
\nc{\rhom}{\mathrm{RHom}}
\nc{\mor}{\mathrm{Mor}}
\nc{\spec}{\mathrm{Spec }}
\nc{\Hom}{\mathrm{Hom}}
\nc{\spc}{\mathrm{Spc}}
\nc{\pt}{\mathrm{pt}}
\nc{\cl}{\mathrm{Cl}}
\nc{\im}{\mathrm{im}}
\nc{\cone}{\mathrm{cone}}
\nc{\perf}{\mathrm{perf}}
\nc{\loc}{\mathrm{Loc}}
\begin{document}

\title[Reconstruction of quasi-compact quasi-separated schemes]{Reconstruction of quasi-compact quasi-separated schemes from categories of perfect complexes}%
\author{Stella Anevski}
\address{} \email{}

\date{}

\maketitle
\begin{abstract}
We show how to recover the underlying topological space of a quasi-compact quasi-separated scheme from the tensor triangulated structure on its category of perfect complexes. 
\end{abstract}
\section*{Introduction} \noindent It has been known for some time that the topology of a noetherian scheme is encoded in the tensor triangulated category of perfect complexes on the scheme. In \cite{balmer}, P. Balmer defines the {\em spectrum} of a tensor triangulated category, and shows that this construction recovers the underlying topological space of a noetherian scheme from its category of perfect complexes.\\
 \\
In this note, we show that the hypothesis of noetherianness can be relaxed at the cost of considering a slightly more complicated reconstruction procedure. Our main result is Theorem \ref{slutsats}, which states that the underlying topological space of a quasi-compact and quasi-separated scheme $(X,{\mathcal O}_X)$ can be reconstructed up to homeomorphism as a certain space of subsets of thick subcategories of $D^{\perf}(X)$, the tensor triangulated category of perfect complexes on $(X,{\mathcal O}_X)$.\\
 \\
The exposition is organized as follows. In $\S 1$, we establish the properties of perfect complexes that are needed in proving the reconstruction result. Inspired by M. B\"okstedt (cf. \cite{neeman}), we construct a map ${\mathcal Loc}$, from tensor triangulated categories to lattices in $\S 2$. When applied to the category $D^{\perf}(X)$, it turns out that this map recovers the lattice ${\mathcal U}(X)$ of open subsets of $X$ (cf. Theorem \ref{faith}). In $\S 3$, we define a map $\pt$, from lattices to topological spaces, and show that it gives a homeomorphism $\pt ({\mathcal U}(X))\simeq X$, for any sober topological space $X$ (cf. Proposition \ref{sober} $(ii)$). The map which recovers the quasi-compact and quasi-separated scheme $(X,{\mathcal O}_X)$ from $D^{\perf}(X)$ is the composition $\pt\circ {\mathcal Loc}$ (cf. Theorem \ref{slutsats}).\\
\section{Preliminaries on perfect complexes}
\noindent Let $(X,{\mathcal O}_X)$ be a quasi-compact and quasi-separated scheme. Recall that a complex $C$ of sheaves of ${\mathcal O}_X$-modules is {\em perfect} if $C$ is locally quasi-isomorphic to some bounded complex of locally free ${\mathcal O}_X$-modules of finite type. The derived category of the category of ${\mathcal O}_X$-modules carries a natural tensor triangulated structure, which descends to the full subcategory $D^{\perf}(X)$ on the perfect complexes. \\
 \\
We define the {\em support} of a complex $C\in D^{\perf}(X)$ to be the closed subset 
$$\begin{array}{rl} 
\supph C= \{x\in X; C_x\not\simeq 0\},
\end{array}\,$$
of those points $x\in X$ at which the stalk complex $C_x$ of ${\mathcal O}_{X,x}$-modules is not acyclic.
\begin{lem} \label{quicpt}
For any commutative unital ring $R$, and any $C\in D^{\perf}(\spec R)$, the set 
$$\begin{array}{rl} 
\spec R\setminus \supph C
\end{array}\,$$
is quasi-compact.
\end{lem}
\begin{proof}
Let $S\subset R$ be a noetherian subring and $C'\in D^{\perf}(\spec S)$ a complex such that $C=C'\otimes_{S}R$. Let $\varphi:\spec R\rightarrow\spec S$ be the natural map. Then $\supph C=\varphi^{-1}(\supph C')$ corresponds to a finitely generated ideal $(f_1,...,f_n)\subset R$. Thus
$$\begin{array}{rl} 
\spec R\setminus \supph C=\bigcup_{i=1}^{n}\spec R[f_{i}^{-1}]
\end{array}\,$$
is a finite union of quasi-compacts.
\end{proof}
\begin{lem} \label{triangles}
Let $a:A\to {\mathcal O}_X$ be a morphism in $D^{\perf}(X)$. If ${\mathcal A}$ is a thick triangulated subcategory of $D^{\perf}(X)$, and $B\in D^{\perf}(X)$ is such that $B\otimes \cone(a)\in {\mathcal A}$, then
$$\begin{array}{rl} 
B\otimes \cone(\otimes^n a)\in {\mathcal A},
\end{array}\,$$
for all $n\geq 1$.
\end{lem}
\begin{proof}
One has exact triangles
$$\vcenter{\xymatrix{
 A\ar[rr]^{a}
& & {\mathcal O}_X\ar[dl]^{}  \\
&\cone(a)  \ar[ul]|{\circ} & }} \,,$$
$$\vcenter{\xymatrix{
 \otimes^n A\ar[rr]^{\otimes^n a}
& & {\mathcal O}_X\ar[dl]^{}  \\
&\cone(\otimes^n a)  \ar[ul]|{\circ} & }} \,,$$
$$\vcenter{\xymatrix{
A\otimes (\otimes^n A)\ar[rr]^{1_A\otimes (\otimes^n a)}
& & A\ar[dl]^{}  \\
&A\otimes \cone(\otimes^n a)  \ar[ul]|{\circ} & }} \,.$$
Since $\otimes^{n+1}a$ can be identified with the composition $a\circ (1_A\otimes (\otimes^n a))$, one obtains the exact triangle
$$\vcenter{\xymatrix{
A\otimes \cone(\otimes^n a)\ar[rr]^{}
& & \cone(\otimes^{n+1} a)\ar[dl]^{}  \\
& \cone(a)  \ar[ul]|{\circ} & }} \,$$
from the octahedral axiom. Using the exact triangle resulting from tensoring this one with $B$, one obtains the desired result by induction on $n$.
\end{proof}
\begin{lem} \label{nilpotence}
Let $(X,{\mathcal O}_X)$ be a quasi-compact and quasi-separated scheme. Let $A,B\in D^{\perf}(X)$, and let $a:A\to {\mathcal O}_X$ be a morphism in $D^{\perf}(X)$. If $a\otimes k(x)=0$ in the derived category $D(k(x))$ of $k(x)$-modules, for all $x\in \supph B$, then there is an $n\geq 1$ such that $1_B\otimes (\otimes^n a)=0$ in $D^{\perf}(X)$.
\end{lem}
\begin{proof}
This follows directly from Theorem $3.8$ in \cite{thomason}.
\end{proof}
\begin{prop} \label{thelemma}
Let $(X,{\mathcal O}_X)$ be a quasi-compact and quasi-separated scheme, and let $B,D\in D^{\perf}(X)$ be such that
$$\begin{array}{rl} 
\supph B\subset \supph D.
\end{array}\,$$
Then $B$ is in the thick triangulated subcategory generated by $D$.
\end{prop}
\begin{proof}
Let $\langle D\rangle$ denote the thick triangulated subcategory generated by $D$. Since $D$ is perfect, there is an isomorphism $\rhom(D,{\mathcal O}_X)\otimes D\simeq \rhom(D,D)$, so $\rhom(D,D)\in \langle D\rangle$. Let 
$$\begin{array}{rl} 
f:{\mathcal O}_X\to \rhom(D,D)
\end{array}\,$$
be the morphism corresponding to $1_D:D\to D$, and let 
$$\begin{array}{rl} 
a:A\to {\mathcal O}_X
\end{array}\,$$
be the edge opposite to the vertex $\rhom(D,D)$ in the exact triangle having $f$ as an edge. Then $\cone(a)\simeq \rhom(D,D)\in \langle D\rangle$, so $B\otimes \cone(a)\in \langle D\rangle$. By Lemma \ref{triangles}, we have
$$\begin{array}{rl} 
B\otimes \cone(\otimes^n a)\in \langle D\rangle,
\end{array}\,$$
for all $n\geq 1$.\\
Next, we show that there is an $n\geq 1$ such that $1_B\otimes (\otimes^n a)=0$ in $D^{\perf}(X)$. By Lemma \ref{nilpotence}, it suffices to show that for all $x\in \supph B$, we have $a\otimes k(x)=0$ in $D(k(x))$. But if $x\in \supph B$, then $D\otimes k(x)\not\simeq 0$, since $\supph B\subset \supph D$. Thus the map 
$$\begin{array}{rl} 
k(x)\to \rhom_{k(x)}(D\otimes k(x), D\otimes k(x)),
\end{array}\,$$
corresponding to $1_{D\otimes k(x)}: D\otimes k(x)\to D\otimes k(x)$ is non-zero, and so is a split monomorphism in $D^{\perf}(\spec k(x))$. But under the isomorphism $\rhom_{{\mathcal O}_X}(D,D)\otimes k(x)\simeq \rhom_{k(x)}(D\otimes k(x),D\otimes k(x))$ (cf. \cite{sga6}, I $7.1.2$), this split monomorphism corresponds to $f\otimes k(x)$. Since $f\circ a=0$, we get $a\otimes k(x)=0$.\\
Applying $\Hom_{D({\mathcal O}_X)}(-,B)$ to the exact triangle
$$\vcenter{\xymatrix{
 B\otimes (\otimes^n A)\ar[rr]^{1_B\otimes(\otimes^n a)}
& & B\otimes {\mathcal O}_X\ar[dl]^{}  \\
&B\otimes \cone(\otimes^n a)  \ar[ul]|{\circ} }} \,,$$
and considering the associated long exact Puppe sequence, one sees that 
$$\begin{array}{rl} 
B\simeq B\otimes {\mathcal O}_X\to B\otimes \cone(\otimes^n a)
\end{array}\,$$
is a split monomorphism. Hence $B\in \langle D\rangle$, since $\langle D\rangle$ is thick and $B$ identifies with a direct summand of $B\otimes \cone(\otimes^n a)\in \langle D\rangle$.
\end{proof}
\section{Recovering open subsets from perfect complexes}
\noindent 
Let ${\mathcal K}$ be a tensor triangulated category. A thick subcategory ${\mathcal A}$ of ${\mathcal K}$ is said to be {\em principal} if it is the triangulated subcategory generated by an element $C\in {\mathcal A}$. If this is the case, we use the notation ${\mathcal A}=\langle C\rangle$. We denote the set of principal thick subcategories of ${\mathcal K}$ by $PS({\mathcal K})$.\\
 \\
A subset $S\subset PS({\mathcal K})$ is said to be {\em filtering} if \\
 $(FS1)$ $\langle C\rangle\in S$ and $\langle C\rangle \subset \langle D\rangle$ implies $\langle D\rangle\in S$ for all $\langle D\rangle\in PS({\mathcal K})$, and\\
 $(FS2)$ for all $\langle C\rangle, \langle D\rangle\in S$, there exists a $\langle B\rangle \in S$ such that $\langle B\rangle \subset \langle C\rangle$ and $\langle B\rangle \subset \langle D\rangle$.\\
We denote the set of filtering subsets of $PS({\mathcal K})$ by ${\mathcal Loc}({\mathcal K})$
\begin{theor} \label{faith}
If $(X,{\mathcal O}_X)$ is a quasi-compact and quasi-separated scheme, then there is an inclusion preserving bijection
$$\begin{array}{rl} 
{\mathcal U}(X)\leftrightarrow  {\mathcal Loc}(D^{\perf}(X))
\end{array}\,$$
between the set of open subsets of $X$ and the set of filtering subsets of $PS(D^{\perf}(X))$.
\end{theor}
\begin{proof}
Consider the maps
$$\begin{array}{rl} 
{\mathcal Loc}(D^{\perf}(X)) &\stackrel{f}{\rightarrow} {\mathcal U}(X),\\
S &\mapsto X\setminus \bigcap_{\langle C\rangle\in S}  \supph  C,
\end{array}\,$$
and
$$\begin{array}{rl}
{\mathcal U}(X) &\stackrel{g}{\rightarrow}  {\mathcal Loc}(D^{\perf}(X)),\\
U &\mapsto \{\langle C\rangle; X\setminus U \subset \supph C\}.
\end{array}\,$$
Note first that $f\circ g=\iden$, since given $x\in U$ there is a $C\in D^{\perf}(X)$ with $X\setminus U\subset \supph C$ but $x\notin \supph C$. Indeed, let $\spec R_x$ be an open affine neighborhood of $x$ such that $\spec R_x\subset  U$, and choose $C$ with $\supph C=X\setminus \spec R_x$ (cf. Lemma $3.4$ in \cite{thomason}).\\
In order to show that $g\circ f=\iden$, let $S\in {\mathcal Loc}(D^{\perf}(X))$ and $D\in D^{\perf}(X)$ be such that
$$\begin{array}{rl} 
\bigcap_{\langle C\rangle\in S}\supph C \subset \supph D.
\end{array}\,$$
We show that $\langle D\rangle\in S$. By Lemma \ref{quicpt}, and since $X$ is quasi-compact, the subset $X\setminus \supph D$ is quasi-compact. By assumption 
$$\begin{array}{rl} 
\{X\setminus \supph C\}_{\langle C\rangle\in S} 
\end{array}\,$$
form an open cover of $X\setminus \supph D$, so there are finitely many $C_i$'s such that
$$\begin{array}{rl} 
\bigcap_{i=1}^{n}\supph C_i \subset \supph D.
\end{array}\,$$
By $(FS2)$, there is a $\langle B\rangle\in S$ such that $\supph B\subset \supph D$. By Proposition \ref{thelemma}, $\langle B\rangle \subset \langle D\rangle$, so by $(FS1)$, $\langle D\rangle\in S$.
\end{proof}
\section{Recovering points from open subsets}
\noindent Recall that a {\em lattice} is a partially ordered set ${\mathcal L}$ which, considered as a category, has all binary products $\wedge$, and all binary coproducts $\vee$. In this section, we shall assume in addition, that lattices have initial and final objects, and that they satisfy the distributive law:
$$\begin{array}{rl} 
U\wedge (V\vee W)=(U\wedge V)\vee (U\wedge W),& \mbox{for all $U,V,W\in {\mathcal L}$}.
\end{array}\,$$
\begin{defi}
Let ${\mathcal L}$ be a lattice with all finite products and all coproducts. A {\em point} of ${\mathcal L}$ is a map
$$\begin{array}{rl} 
p:{\mathcal L}\to \{0,1\},
\end{array}\,$$
of partially ordered sets, which preserves finite products and infinite coproducts. The set of points of a lattice ${\mathcal L}$ will be denoted by $\pt ({\mathcal L})$. 
\end{defi}
\noindent We define a topology on $\pt ({\mathcal L})$ by declaring all subsets which are of the form
$$\begin{array}{rl} 
\{p;\mbox{    } p(U)=1\}, &\mbox{for some $U\in {\mathcal L}$,}
\end{array}\,$$
to be open.\\
 \\
The set $\pt({\mathcal L})$ is in bijection with the set of {\em proper primes} of ${\mathcal L}$, that is, the set of elements $P\in {\mathcal L}$ which are not final in ${\mathcal L}$, and which satisfy
$$\begin{array}{rl} 
\mbox{$U\wedge V\leq P$ if and only if $U\leq P$ or $V\leq P$,  for all $U,V\in {\mathcal L}$}.
\end{array}\,$$
Under this bijection, a point $p\in \pt({\mathcal L})$ corresponds to the coproduct
$$\begin{array}{rl} 
\bigvee_{p(W)=0} W,
\end{array}\,$$
of elements in the kernel of $p$, and a proper prime $P$ corresponds to the map
$$\begin{array}{rl} 
p:{\mathcal L}\to \{0,1\}, & \mbox{$p(U)=0$ if and only if $U\leq P$}.
\end{array}\,$$
\begin{exem}
Let $X$ be a topological space, and consider the lattice ${\mathcal U}(X)$, of open subsets of $X$. A point $x\in X$ determines the point
$$\begin{array}{rl} 
p:{\mathcal U}(X)\to \{0,1\}, & \mbox{$p(U)=0$ if and only if $x\notin U$}.
\end{array}\,$$
Under the bijection with proper primes of ${\mathcal U}(X)$, this point corresponds to 
the element 
$$\begin{array}{rl} 
X\setminus \overline{\{x\}}\in {\mathcal U}(X).
\end{array}\,$$
\end{exem}
\begin{defi}
A topological space $X$ is {\em sober} if every non-empty irreducible closed set $Z$ has a unique generic point, that is, there exists a $z\in Z$ such that
$$\begin{array}{rl} 
Z=\overline{\{z\}}.
\end{array}\,$$
\end{defi}
\noindent Examples of sober spaces include Hausdorff spaces and topological spaces underlying schemes.
\begin{prop} \label{sober}
Let $X$ be a topological space, and let $F$ be the map $x \mapsto X\setminus \overline{\{x\}}$, from $X$ to the set of proper primes of ${\mathcal U}(X)$.
$$\begin{array}{rl} 
(i) &\mbox{The map $F$ is a bijection  if and only if $X$ is sober.}\\
(ii) &\mbox{If $X$ is sober, then $F$ induces a homeomorphism between $X$ and }\\
&\mbox{the space $\pt({\mathcal U}(X))$ of points of ${\mathcal U}(X)$.}
\end{array}\,$$
\end{prop}
\begin{proof}
$(i)$ Note that an open set $U\in {\mathcal U}(X)$ is a proper prime if and only if its complement $X\setminus U$ is non-empty and irreducible.\\
$(ii)$ Follows from $(i)$, and the fact that the image of an open set $U\subset X$ gets identified with the open set
$$\begin{array}{rl} 
\{p; \mbox{    } p(U)=1\}\subset \pt({\mathcal U}(X)).
\end{array}\,$$
\end{proof}
\noindent Combining Theorem \ref{faith} and Proposition \ref{sober}, we obtain the announced reconstruction result:
\begin{theor} \label{slutsats}
Let $(X,{\mathcal O}_X)$ be a quasi-compact and quasi-separated scheme. Then the underlying topological space $X$ is homeomorphic to the space $\pt\circ {\mathcal Loc}(D^{\perf}(X))$, of points of the lattice of filtering subsets of principal thick subcategories of $D^{\perf}(X)$.
\end{theor}

\section*{Acknowledgments}
\noindent The author would like to thank Ryszard Nest. The author was supported by the Danish National Research Foundation (DNRF) through the Centre for Symmetry and Deformation. 


\begin{thebibliography}{99}
\bibitem{balmer} {\sc P. Balmer}, ``The spectrum of prime ideals in tensor triangulated categories'', J. Reine Angew. Math. {\bf 588}:
149-.168 (2005).
\bibitem{sga6} {\sc P. Berthelot, A. Grothendieck, L. Illusie}, ``Th\'eorie des Intersections et Th\'eor\`eme de Riemann-Roch'', Springer-Verlag Lecture Notes in Mathematics {\bf 225} (1971).
\bibitem{neeman} {\sc A. Neeman}, ``The chromatic tower for $D(R)$'', Topology {\bf 31} no.3, 519-.532 (1992).
\bibitem{thomason} {\sc R.W. Thomason}, ``The classification of triangulated subcategories'', Compositio Mathematica {\bf 105} 1-.27 (1997).
\end{thebibliography}
\end{document}